\documentclass{amsart}
\usepackage{enumerate}

\newtheorem{theorem}{Theorem}[section]
\newtheorem{lemma}[theorem]{Lemma}

\theoremstyle{definition}

\theoremstyle{remark}
\newtheorem{remark}[theorem]{Remark}
\def\Xint#1{\mathchoice
	{\XXint\displaystyle\textstyle{#1}}%
	{\XXint\textstyle\scriptstyle{#1}}%
	{\XXint\scriptstyle\scriptscriptstyle{#1}}%
	{\XXint\scriptscriptstyle\scriptscriptstyle{#1}}%
	\!\int}
\def\XXint#1#2#3{{\setbox0=\hbox{$#1{#2#3}{\int}$ }
		\vcenter{\hbox{$#2#3$ }}\kern-.6\wd0}}

\numberwithin{equation}{section}

\begin{document}

\title{Quasi-positive curvature and projectivity}

\author{Yiyang Du}
\address{ School of Mathematical Sciences, Capital Normal University, Beijing 100048, China}
\email{18843409130@163.com}

\author{Yanyan Niu}
\address{School of Mathematical Sciences, Capital Normal University, Beijing 100048, China}
\email{yyniu@cnu.edu.cn}
\thanks{The second author was supported by National Natural Science Foundation of China \#11821101.}


\date{October 23, 2024}



\begin{abstract}
		In this paper, we first prove that a compact K\"ahler manifold is projective if it satisfies certain quasi-positive curvature conditions, including quasi-positive  $S_2^\perp,\, S_2^+,\,\mbox{Ric}_3^\perp, \,\mbox{Ric}_3^+$ or $2$-quasi-positive $\mbox{Ric}_k$. Subsequently, we prove that a compact K\"ahler manifold with a restricted holonomy group is both projective and rationally conected if it satisfies some non-negative curvature condition, including non-negative  $S_2^\perp,\, S_2^+,\,\mbox{Ric}_3^\perp, \,\mbox{Ric}_3^+$ or $2$-non-negative $\mbox{Ric}_k$.

\end{abstract}

\maketitle

\section{Introduction}
There are many results on the projectivity and rational connectedness of compact K\"ahler manifolds related to positive or quasi-positive curvature conditions. Here quasi-positive means non-negative everywhere and positive somewhere in the manifold. A projective manifold is called rationally connected if any two points in the manifold can be joined by a rational curve. Let $(M^n, g)$ be a $n$-dimensional compact K\"ahler manifold. For $x\in M$, the holomorphic tangent space is denoted by $T'_xM$. For $X\in T_x'M,\, \mbox{Ric}(X)$ and $H(X)$ denote the Ricci curvature and the holomorphic sectional curvature in the direction of $X$ respectively. The projectivity of a compact K\"ahler manifold is equivalent to the existence of a positive holomorphic line bundle by Kodaira embedding theorem \cite{Kodaira}, which implies that any compact K\"ahler manifold $(M, g)$ with $H^{2,0}_{\bar{\partial}}(M, \mathbb{C})=0$ must be projective. 
Then by using Bochner formula, it is easy to obtain that any compact K\"ahler manifold with positive Ricci curvature ($\mbox{Ric}>0$) is projective. The rational connectedness of 
such manifolds (with $\mbox{Ric}>0$) was established by Campana\cite{Cam} and Koll\'ar-Miyaoka-Mori\cite{KMM}.
In \cite{Yang20}, Yang showed that any compact K\"ahler manifold with quasi-positive Ricci curvature is projective and rationally connected using integration arguments.

	On the holomorphic sectional curvature, Yau\cite{Yau} proposed a well-known conjecture which stated that any K\"ahler manifold with positive holomorphic sectional curvature ($\mbox{H}>0$) is projective and rationally connected. Yau's conjecture was completely confirmed by Yang\cite{YangRC} by considering $\mbox{RC}$-positivity for Hermitian vector bundle and a minimum principle. In fact, Heier-Wong \cite{HW} previously proved that any projective manifold with quasi-positive holomorphic sectional curvature is rationally connected using average arguments and certain integration by parts. Recently, Zhang-Zhang\cite{ZZ} extended Yau's conjecture to the quasi-positive case, with the key point in their proof being an integral inequality.
	
		There has been interest in studying K\"ahler manifolds with some curvature conditions interpolating between the holomorphic sectional curvature and the Ricci curvature.  The $k$-Ricci curvature $\mbox{Ric}_k$ was initiated in the study of $k$-hyperbolicity of a compact K\"ahler manifold by Ni \cite{NiCPAM}. $\mbox{Ric}_k$ is defined as the Ricci curvature of the $k$-dimensional holomorphic subspace of the holomorphic tangent bundle $T'M$. It coincides with the holomorphic sectional curvature when $k=1$, and with the Ricci curvature when $k=n$.  In \cite{Nicrell}, Ni showed that any compact  K\"ahler manifold $(M^n, g)$ with positive $\mbox{Ric}_k$ for some $1\le k\le n$ is projective and rationally connected by applying the maximum principle via viscosity consideration. In the proof of projectivity, Whitney's comass is employed. The proof of rationally connected requires a second variation consideration.  A more generalized curvature, mixed curvature,  was introduced by Chu-Lee-Tam\cite{CLT}. It is defined as a linear combination of $\mbox{Ric}$ and $H$, which is defined as follows:
	$$
	C_{a,b}(x,X)=a\mbox{Ric}(X, \bar{X})+bH(X)
	$$
	for $X\in T_x'M$ with $|X|=1$ and $a, b\in \mathbb{R}$.  If for any $x\in M,\, X\in T_x'M,\, |X|=1$, $C_{a,b}(x, X)\ge 0$ (resp. $>0$), we say $C_{a, b}(x)\ge 0$ (resp. $>0$). If $C_{a, b}(x)\ge 0$, and $C_{a, b}(x_0)>0$ at some point $x_0$, we say $C_{a, b}$ is quasi-positive. It is easy to see that $C_{1, 0}$ is the standard Ricci curvature, $C_{0, 1}$ is the holomorphic sectional curvature, $C_{1, -1}$ is the orthogonal Ricci curvaure $\mbox{Ric}^\perp$ 
	introduced in\cite{NZ-CV} and $C_{1, 1}$ is $\mbox{Ric}^+$ introduced in Ni\cite{Nicrell}. In addition,
	$\mbox{Ric}_k$ is related with $C_{k-1,n-k}$ by Lemma 2.1 in \cite{CLT}. It was shown that any compact K\"ahler manifold with $C_{a, b}>0$ is projective  when $a>0, 3a+2b\ge 0$ and is simply  connected when $a>0$ and $a+b\ge 0$ in \cite{BT,CLT}. Obviously, it includes the results when $\mbox{Ric}^\perp>0, \mbox{Ric}^+>0,$ or $\mbox{Ric}_k>0$  proved in \cite{Nicrell,NZ-CV}.
	Very recently, Tang\cite{Tang} extends the results for positive mixed curvature to the quasi-positive case, and showed that  any compact K\"ahler manifold with quasi-positive $\mathcal{C}_{a,b}$ must be projective when  $a\ge 0, 3a+2b\ge 0$ and rationally connected when $a, b\ge 0$.  
	Chu-Lee-Zhu \cite{CLZ} proved that any compact K\"ahler manifold with quasi-positive $\mathcal{C}_{a,b}$ must be simply connected when $a>0, a+b>0$ by adapting the conformal pertubation method. 
	
		In \cite{Nicrell, NZ-CV,NZ-GT}, the projectivity of a compact K\"ahler manifold was conformed under other positive curvature conditions by considering the $\partial\bar{\partial}$-Bochner formula or $\partial\bar{\partial}$-operator on the comass of holomorphic $(p,0)$-forms and maximum principle. These curvature conditions  include the positivity of $S_2, \, S_2^\perp,\, S_2^+, \mbox{Ric}_k^\perp, $ $\mbox{Ric}_k^+$, and $2$-positivity of $\mbox{Ric}_k$ for some $1\le k\le n$,  even much weaker $BC$-2 positive curvature.  In \cite{Tang}, Tang also extended the result for positive $S_2$ to the quasi-positive case, proving that any compact K\"ahler manifold with quasi-positive $S_2$ must be projective by applying integration arguments and a key identity given in \cite{ZZ}. In this paper, we will generalize the results on positive $S_2^\perp, S_2^+, \mbox{Ric}_3^\perp, \mbox{Ric}_3^+, 2$-positive  $\mbox{Ric}_k$ to the quasi-positive cases by adapting similar arguments as in \cite{Tang}. In summery, our main result is the following:
		
			\begin{theorem}\label{thm1}
			Let $(M^n, g)\,(n\ge 2)$ be a compact K\"ahler manifold with one of the following curvature conditions,
			\begin{enumerate}
				\item quasi-positive $S_2^\perp$ or $S_2^+$;
				\item quasi-positive $\mbox{Ric}_3^\perp$ or $\mbox{Ric}_3^+$ when $n\ge 3$;
				\item  $2$-quasi-positive $\mbox{Ric}_p$ for some $2\le p\le n$.
			\end{enumerate}  
			Then $h^{2,0}=0$. In particular, $M$ is projective. 
		\end{theorem} 
		Theorem \ref{thm1} partially answers the question asked by Ni in \cite{Nihol}, which said that whether or not the projectivity result under the $BC$-2 positivity in \cite{Nicrell} still holds when the positivity is replaced by quasi-positivity. 
		
			For the non-negative curvature conditions, additional requirements are necessary to guarantee that the K\"ahler manifold is projective. Chu-Lee-Zhu\cite{CLZ} showed that a compact K\"ahler manifold with non-negative mixed curvature $C_{a, b}\ge 0$ where $a>0, a+b>0$, and holonomy group $U(n)$ is projective. Ni\cite{Nihol} extended this result to the case where $a>0, 3a+2b>0$. 
		Ni also proved that a compact K\"ahler manifold with $S_2\ge 0$ (or with the smallest two eigenvalues of $\mbox{Ric}$ being non-negative) and  holonomy gourp $U(n)$ is not only projective but also rationally connected. The second main result of this paper  extends the results in \cite{Nihol} to the curvature conditions considered in Theorem \ref{thm1}.
		\begin{theorem}\label{thm2}
			Let $(M^n, g)\, (n\ge 2)$ be a compact K\"ahler manifold with holonomy group $U(n)$. Assume one the the following holds:
			\begin{enumerate}[(a)]
				\item $S_2^\perp\ge 0$ or $S_2^+\ge 0$;
				\item $\mbox{Ric}_3^\perp\ge 0$ or $\mbox{Ric}_3^+\ge 0$ when $n\ge 3$;
				\item  $2$-non-negative $\mbox{Ric}_p$ for some $2\le p\le n$.
			\end{enumerate}
			Then $h^{2,0}=0$. In particular, $M$ is projective and rationally connected.  
		\end{theorem}
		Here, $2$-non-negative $\mbox{Ric}_p$ is equivalent to that $S_2\ge 0$ when $p=2$ and that the smallest two eigenvalues of $\mbox{Ric}$ is non-negative when $p=n$. 
		
			\section{preliminaries}
		
		In this section, we will present the curvature conditions introduced in \cite{NiCPAM,Nicrell,NZ-CV,NZ-GT} and useful facts that are necessary for our results. 
		\subsection{ Curvature conditions} 
		
		Let $(M^n, g)$ be a K\"ahler manifold of complex dimension $n\ge 2$. Let $\{z^i\}_{i=1}^n$ be the local holomorphic coordinates near $x$ on $M$. At $x$, the curvature tensor components are 
		$$
		R_{i\bar{j}k\bar{l}}
		=-\frac{\partial^2 g_{i\bar{j}}}{\partial z^k\partial\bar{z}^l}+g^{p\bar{q}}\frac{\partial g_{i\bar{q}}}{\partial z^k}\frac{\partial g_{p\bar{j}}}{\partial \bar{z}^l}.
		$$
		
		The Ricci curvature  is  $$\mbox{Ric}(X, \bar{X})=\sum_{i,j}g^{i\bar{j}}R(X, \bar{X}, \frac{\partial}{\partial z^i},\frac{\partial}{\partial\bar{z}^j}),$$
		for any $X\in T_x'M$. 
		
			The holomorphic sectional curvature $H(X)$ is 
		$$
		H(X)=\frac{R(X,\bar{X},X,\bar{X})}{|X|^4}
		$$
		for any nonzero $X\in T_x'M$. 
		
		In \cite{NZ-CV}, the orthogonal Ricci curvature $\mbox{Ric}^\perp$ is defined as the difference of the Ricci curvature and the holomorphic sectional curvature, that is,
		$$
		\mbox{Ric}^\perp(X, \bar{X}) = \mbox{Ric}(X, \bar{X})-|X|^2 H(X),
		$$
		for any nonzero $X\in T_x'M$. 
		
		In \cite{Nicrell}, $\mbox{Ric}^+$ is defined as the sum of the Ricci curvature and the holomorphic sectional curvature, that is,
		$$
		\mbox{Ric}^+(X, \bar{X}) = \mbox{Ric}(X, \bar{X})+|X|^2 H(X),
		$$
		for any nonzero $X\in T_x'M$. 
		
			In \cite{NiCPAM}, they introduced $k$-Ricci curvature $\mbox{Ric}_k$ which interpolates between holomorphic sectional curvature and Ricci curvature. 
		Let $\Sigma$ be a $k$-subspace in $T_x'M$, $\mbox{Ric}_k(x, \Sigma)$ is the Ricci curvature of $\Sigma$. For any $v\in \Sigma$,  
		$$
		\mbox{Ric}_k(x, \Sigma)(v,\bar{v})=\sum_{j=1}^k R(v,\bar{v}, E_j,  \bar{E}_j),
		$$
		where $\{E_i\}_{i=1}^n$ is any unitary orthogonal frame of $\Sigma$.
		We say $\mbox{Ric}_k(x)>0\, (\ge 0)$ if and only if $\mbox{Ric}_{k,-}(x, \Sigma)>0 \,(\ge 0)$ for any $k$-subspace $\Sigma\subset T_x'M$, where $\mbox{Ric}_{k,-}(x, \Sigma)=\inf_{v\in \Sigma, |v|=1}\limits \mbox{Ric}_k(x, \Sigma)(v, \bar{v})$. If the sum of the lowest two eigenvalues of $\mbox{Ric}_k(x, \Sigma)$ is positive (resp.\,non-negative), we say $\mbox{Ric}_k(x,\Sigma)$ is $2$-positive (resp.\,$2$-non-negative). Obviously, $\mbox{Ric}_1$ corresponds to the holomorphic sectional curvature. When $k\ge 2$, we say $\mbox{Ric}_k(x)$ is $2$-positive if for any $x$ and any $k$-subspace $\Sigma\subset T_x'M$, $\mbox{Ric}_k(x, \Sigma)$ is $2$-positive.  If $\mbox{Ric}_k(x)$ is $2$-non-negative and at least there is some point $x_0\in M$ such that $\mbox{Ric}_k(x_0)$ is $2$-positive, we say $\mbox{Ric}_k(x)$ is 2-quasi-positive.  
		
			In \cite{NZ-GT}, the $k$-scalar curvature $S_k(x, \Sigma)$ is defined as
		$$
		S_k(x, \Sigma)=k\Xint{-}_{z\in \Sigma,\, |z|=1}\mbox{Ric}_k(Z, \bar{Z})d\theta(Z),
		$$
		where $\Xint{-} f(Z)d\theta(Z)=\frac{1}{\mbox{Vol}(\mathbb{S}^{2k-1})}\int_{\mathbb{S}^{2k-1}}f(Z)d\theta(Z)$ and $\mathbb{S}^{2k-1}$ is the unit sphere in $\Sigma$. It follows from the Berger's average trick that $S_k(x, \Sigma)$ is the trace of $\mbox{Ric}_k(x, \Sigma)$, that is, if $E_1,\cdots, E_k$ are unitary orthogonal frame of $\Sigma$, then 
		$$
		S_k(x, \Sigma)=\sum_{j=1}^k \mbox{Ric}_k(x, \Sigma)(E_j, \bar{E}_j).
		$$
		Let $S_k(x)=\inf_\Sigma\limits S_k(x, \Sigma)$. 
		We say $S_k(x)>0 \,(\ge 0)$ if and only if for any $k$-subspace $\Sigma\subset T_x'M, S_k(x,\Sigma)>0 \,(\ge 0)$. $S_k(x)$ is quasi-positive if $S_k(x)\ge 0$ and $S_k(x_0)>0$ at some point $x_0\in M$. It is easy to see that the positivity of $S_k$ implies the positivity of $S_l$ when $l\ge k$ and if $\mbox{Ric}_k$ is $2$-quasi-positive for some $2\le k\le n$, then $S_3$ is quasi-positive and $S_l$ is quasi-positive for any $3\le l\le n$. 
		
		In \cite{Nicrell,NZ-CV}, they also introduced the curvature $S_k^\perp (S_k^+)$. For any $k$-dimensional subspace $\Sigma\subset T_x'M$,  
		$$
		S_k^\perp(x, \Sigma)=k\Xint{-}_{z\in \Sigma, \, |z|=1}\mbox{Ric}^\perp(Z, \bar{Z})d\theta(Z),
		$$
		$$
		S_k^+(x, \Sigma)=k\Xint{-}_{z\in \Sigma, \, |z|=1}\mbox{Ric}^+(Z, \bar{Z})d\theta(Z).
		$$
		They are related to $S_k(x, \Sigma)$ and Ricci curvature by the following identities
		\begin{equation}\label{eq:Sperp}
			S_k^\perp(x, \Sigma)=\sum_{j=1}^k\mbox{Ric}(E_j, \bar{E}_j)-\frac{2}{k+1}S_k(x,\Sigma),
		\end{equation}
		\begin{equation}\label{eq:Ssum}
			S_k^+(x, \Sigma)=\sum_{j=1}^k\mbox{Ric}(E_j, \bar{E}_j)+\frac{2}{k+1}S_k(x,\Sigma),
		\end{equation}
		where $E_1,\cdots, E_k$ are unitary orthogonal frame of $\Sigma$.
		Let $S_{k}^\perp(x)=\inf_\Sigma\limits S_k^\perp(x, \Sigma)$ (resp. $S_{k}^+(x)=\inf_\Sigma\limits S_k^+(x, \Sigma)$ ). 	Hence, $S_k^\perp(x)$ interpolates between $\mbox{Ric}^\perp(X, \bar{X})$ and $\frac{n-1}{n+1}S(x)$,  $S_k^+(x)$ interpolates between $\mbox{Ric}^+(X, \bar{X})$ and $\frac{n+3}{n+1}S(x)$.  We say $S_k^\perp(x)>0 (\ge 0)$ if and only if for any $k$-subspace $\Sigma\subset T_x'M,\, S_k^\perp(x,\Sigma)>0 \,(\ge 0)$. $S_k^\perp(x)$ is quasi-positive if $S_k^\perp\ge 0$ and $S_k^\perp(x_0)>0$ at some point $x_0\in M$. The same definitions of positivity, non-negativity, and quasi-positivity apply to $S_k^+$. It is also easy to see that if $S_2^\perp \, (\mbox{resp}.\, S_2^+)$ is quasi-positive,  $S_l^\perp \, (\mbox{resp}.\, S_l^+)$ is quasi-positive when $l\ge 2$.
		
			In \cite{Nicrell}, Ni also considered the following curvatures
		\begin{align*}
			\mbox{Ric}^\perp_k(x, \Sigma)(v, \bar{v})&=\mbox{Ric}_k(x, \Sigma)(v,\bar{v})-\frac{R(v, \bar{v}, v,\bar{v})}{|v|^2},\\
			\mbox{Ric}^+_k(x, \Sigma)(v, \bar{v})&=\mbox{Ric}_k(x, \Sigma)(v,\bar{v})+\frac{R(v, \bar{v}, v,\bar{v})}{|v|^2}
		\end{align*}
		for any nonzero $v\in \Sigma$. We say $\mbox{Ric}^\perp_k(x)>(\ge)\, 0$  if and only if for any $\Sigma\subset T_x'M, \, v\in \Sigma\setminus\{0\},\, \mbox{Ric}^\perp_k(x, \Sigma)(v, \bar{v})>(\ge) \,0$.  $\mbox{Ric}^\perp_k(x)$ is quasi-positive if $\mbox{Ric}^\perp_k(x)\ge 0$ and  $\mbox{Ric}^\perp_k(x_0)> 0$ at some point $x_0\in M$. Similarly, one can also define the positivity, non-negativity, quasi-positivity of $\mbox{Ric}_k^+$. If $\mbox{Ric}_k^\perp(x) \,(\mbox{or} \, \mbox{Ric}_k^+(x))$ is quasi-positive, then $S_{k}$ is quasi-positive as shown by the following equalities from
		\cite{Nicrell},
		\begin{align*}
			\Xint{-}_{\mathbb{S}^{2k-1}\subset\Sigma\subset T_x'M} \mbox{Ric}^\perp_{k}(Z,\bar{Z})d\theta(Z)&=\frac{k-1}{k(k+1)}S_k(x,\Sigma),\\
			\Xint{-}_{\mathbb{S}^{2k-1}\subset\Sigma\subset T_x'M} \mbox{Ric}^+_{k}(Z,\bar{Z})d\theta(Z)&=\frac{k+3}{k(k+1)}S_k(x,\Sigma).
		\end{align*}
		In particular, if $\mbox{Ric}_3^\perp \,(\mbox{or} \, \mbox{Ric}_3^+(x))$ is quasi-positive, then $S_l(x)$ is quasi-positive when $l\ge 3$. 

	\subsection{Fundamental identity}
We will recall the key identity for forms in \cite{Tang,ZZ}. The identity is also fundamental in our proof. 
Let $\omega$ be the K\"ahler form of the K\"ahler manifold $(M^n, g)$. Let $\alpha$ be a real $(1, 1)$-form and $\eta$ be a $(p, 0)$-form on $M$. $\beta$ is a real $(1, 1)$-form defined as
$$ 
\beta=\Lambda^{p-1}\left((\sqrt{-1})^{p^2}\frac{\eta\wedge \bar{\eta}}{p!}\right),
$$
where $\Lambda=\star^{-1}L\star$ is adjoint operator of $L$ which is defined as $L\phi=\omega\wedge \phi$ for any $(p, q)$-form $\phi$, $\star$ is the Hodge star operator.  
\begin{lemma}\label{lm:hodge}
	With the assumptions above, we have
	\begin{enumerate}
		\item (\cite{Demailly}) For every $u\in \Lambda^k(T^{\mathbb{R}}M\otimes_{\mathbb{R}} \mathbb{C})$, 
		$$[L^r, \Lambda]u=r(k-n+r-1)L^{r-1}u.$$
		\item $\star \eta=(-\sqrt{-1})^{p^2}\eta\wedge \frac{\omega^{n-p}}{(n-p)!}$.
	\end{enumerate}
	
\end{lemma}

The key identity is the following 
\begin{lemma}[\cite{ZZ}]\label{lm:identity} With the assumptions above, 
	\begin{equation}\label{eq:identity}
		(\sqrt{-1})^{p^2}\alpha\wedge \eta\wedge \bar{\eta}\wedge\frac{\omega^{n-p-1}}{(n-p-1)!}=[\mbox{tr}_\omega\alpha\cdot |\eta|^2_g-p<\alpha, \beta>_g]\frac{\omega^n}{n!}.
	\end{equation}
	
\end{lemma}

\begin{proof}
	Let $x\in M$ and $\{z^1, z^2, \cdots, z^n\}$ be local holomorphic coordinates at $x$ such that $w=\sqrt{-1}\sum_i\limits dz^i\wedge d\bar{z}^i$,  $\alpha=\sqrt{-1}\sum_{i}\limits a_{i\bar{i}} dz^i\wedge d\bar{z}^i$ and $\eta=\sum_{I}\limits \eta_{I}dz^I$, where $I=(i_1<i_2<\cdots<i_p)$. 
	
	Then $\beta=\frac{1}{p}\sqrt{-1}\sum_{j,k,J,|J|=p-1}\limits \eta_{jJ}\overline{\eta_{kJ}}dz^j\wedge d\bar{z}^k$. By using Lemma \ref{lm:hodge} and properties of Hodge $\star$ operator, we have 
	\begin{align*}
		<\alpha\wedge \eta, \omega\wedge \eta>_g\frac{\omega^n}{n!}
		&= \alpha\wedge \eta\wedge \overline{\star L\eta}\\
		&= \alpha\wedge \eta\wedge \overline{\Lambda \star \eta}\\
		&= (\sqrt{-1})^{p^2} \alpha\wedge \eta\wedge \overline{\Lambda(\frac{\omega^{n-p}}{(n-p)!}\wedge \eta)} \\
		&= (\sqrt{-1})^{p^2} \frac{\alpha}{(n-p)!}\wedge \eta\wedge(\overline{\Lambda L^{n-p}\eta)}\\
		&= (\sqrt{-1})^{p^2}\alpha\wedge \eta\wedge \bar{\eta}\wedge \frac{\omega^{n-p-1}}{(n-p-1)!}. 
	\end{align*}
	
	On the other hand, 
	\begin{align*}
		<\alpha\wedge \eta, \omega\wedge \eta>_g
		&= \sum_{I}\sum_{i\notin I}\limits \alpha_{i\bar{i}}|\eta_I|^2\\
		&= \mbox{tr}_\omega \alpha\cdot |\eta|_g^2-p<\alpha, \beta>_g.
	\end{align*}
	Therefore, (\ref{eq:identity}) holds.
\end{proof}

\begin{remark}
	Our proof of Lemma \ref{lm:identity} utilizes the properties of Hodge operators, differing slightly from the proof given in \cite{ZZ}.
\end{remark}

\section{Proof of Theorem \ref{thm1} and Theorem \ref{thm2}}
\subsection{Proof of Theorem \ref{thm1}}
Let ~$(M^n, g)$~be a compact K\"ahler manifold of complex dimension $n\ge 2$.
We will adapt the similar arguments from \cite{Tang} and prove by contradiction. Suppose $h^{2,0}=\mbox{dim}H^{2, 0}_{\bar{\partial}}(M, \mathbb{C})\neq 0$ and let $\xi$ be a nonzero holomorphic $(2,0)$-form on $M$. Let $m$ be the largest integer such that $\eta=\xi^m\neq 0$ but $\xi^{m+1}\equiv 0$. Then at any point $x\in M$ where $\eta(x)\neq 0$, there is a unitary holomorphic frame $\{e_i\}_{i=1}^n$ such that $\omega=\sqrt{-1}\sum_{i=1}^n\limits e^i\wedge \bar{e}^i$ and 
$$
\xi=\lambda_1e^{1}\wedge e^{2}+\cdots+\lambda_m e^{2m-1}\wedge e^{2m}. 
$$ 
Hence,
$$
\eta=m!\lambda_1\cdots\lambda_m e^{1}\wedge \cdots\wedge e^{2m}, \quad \beta=\frac{|\eta|^2}{2m}\sqrt{-1}(e^1\wedge \bar{e}^1+\cdots+e^{2m}\wedge \bar{e}^{2m}).
$$
By the Bochner formula in \cite{NZ-GT}, we have that at $x$, for any $v\in T_{x}'M$,
\begin{equation}\label{eq:bochner}
	\left<\sqrt{-1}\partial\bar{\partial}|\eta|_g^2, \frac{1}{\sqrt{-1}}v\wedge \bar{v}\right>
	=|\nabla_v \eta|_g^2+|\eta|_g^2\sum_{i=1}^{2m}R_{i\bar{i}v\bar{v}}.
\end{equation}
Using (\ref{eq:bochner}), it is straightforward to obtain  
\begin{equation}\label{eq:laplace}
	\Delta_g|\eta|_g^2=\sum_{i=1}^n\limits|\nabla_i\eta|_g^2+|\eta|_g^2\sum_{i=1}^{2m}R_{i\bar{i}}
\end{equation}
and 
\begin{equation}\label{eq:partial}
	<\sqrt{-1}\partial\bar{\partial}|\eta|_g^2, \beta>_g=\frac{|\eta|_g^2}{2m}\sum_{i=1}^{2m}\limits |\nabla_i\eta|_g^2+\frac{|\eta|_g^4}{2m}\sum_{i,j=1}^{2m}R_{i\bar{i}j\bar{j}}.
\end{equation}
As in \cite{Tang} (one can also use identity (\ref{eq:identity})), we have 
\begin{equation}\label{eq:key}
	-\frac{1}{2}\int_M |d|\eta|_g^2|_g^2\cdot\omega^n=\int_{M}\Delta_g|\eta|^2_g\cdot|\eta|_g^2\cdot \omega^n=2m\int_M\langle \sqrt{-1}\partial\bar{\partial}|\eta|_g^2, \beta\rangle_g\cdot\omega^n.
\end{equation}
Now we will divide the proof into three parts. 

	\textbf{ 	(1)  $S_2^\perp$ or $S_2^+$ is quasi-positive.}


If $S_2^\perp$ is quasi-positive, then $S_l^\perp$ is quasi-positive for $2\le l\le n$. 
Let $\Sigma=\mbox{span}\{e^j\}_{j=1}^l$, then (\ref{eq:Sperp}) implies that
$$
S_l^\perp(x.\Sigma)=\sum_{j=1}^l\limits R_{j\bar{j}}-\frac{2}{l+1}S_l(x,\Sigma)
=\sum_{j=1}^l\limits R_{j\bar{j}}-\frac{2}{l+1}\sum_{i,j=1}^l\limits  R_{i\bar{i}j\bar{j}}.
$$
When $l=2m$, together with (\ref{eq:laplace}) and (\ref{eq:partial}), we have 
\begin{eqnarray}
	&\,&	\frac{\Delta_g|\eta|^2_g\cdot|\eta|_g^2}{2m}-
	\frac{2}{2m+1}<\sqrt{-1}\partial\bar{\partial} |\eta|^2_g, \beta>_g\nonumber\\
	&=&\frac{(2m-1)|\eta|_g^2}{2m(2m+1)}\sum_{i=1}^{2m}|\nabla_i\eta|_g^2+\frac{|\eta|_g^2}{2m}\sum_{i=2m+1}^n|\nabla_i\eta|_g^2+\frac{|\eta|_g^4}{2m}S_{2m}^\perp(x, \Sigma)\label{eq:S1}\\
	&\ge & \frac{|\eta|^4}{2m}S_{2m}^\perp(x).\nonumber
\end{eqnarray}
Hence,
$$
0\ge 
-\frac{2m-1}{4m(2m+1)}\int_M |d|\eta|_g^2|_g^2\omega^n\ge 
\int_M \frac{|\eta|_g^4}{2m}S_{2m}^\perp(x)\omega^n\ge 0.
$$
Therefore, $|\eta|\equiv C$ for some constant $C$. Since $S_{2m}^\perp$ is quasi-positive, then 
$$
0\ge \frac{C^4}{2m}\int_M S_{2m}^\perp(x)\omega^n>0,
$$ 
which is a contradiction if $C\neq 0$. Hence, there is no nonzero holomorphic $(2, 0)$-form. That is, $h^{2, 0}=0$ and $M$ is projective.

When $S_2^{+}$ is quasi-positive, we can also conclude that $S_l^+$ is quasi-positive for $2\le l\le n$. Then we can follow almost  the same proof as that for $S_2^\perp$, with only minor modifications. One simply needs to observe the difference between the two curvatures and replaces "$-$" to "$+$" on the left hand side of the equality (\ref{eq:S1}).

	\textbf{ (2) $\mbox{Ric}_3^\perp$ or $\mbox{Ric}_3^+$ is quasi-positive.}    

Then $S_3$ is quasi-positive and $S_l$ is quasi-positive for any $l\ge 3$.
We first deal with the case that $\mbox{Ric}_3^\perp$ is quasi-positive. With the same arguments in \cite{Tang}, one can easily get the contradiction when $m\ge 2$. We complete the proof as follows for convenience of the readers. If $ m\ge 2$, by (\ref{eq:partial}), we have 
\begin{equation}\label{eq:Ric6}
	<\sqrt{-1}\partial\bar{\partial}|\eta|_g^2, \beta>_g \ge \frac{|\eta|^2}{2m}\sum_{i=1}^{2m}\limits |\nabla_i\eta|^2+\frac{|\eta|^4}{2m}S_{2m}(x).
\end{equation}
Then 
\begin{align*}
	-\frac{1}{2}\int_M |d|\eta|_g^2|_g^2\cdot\omega^n=\int_{M}\Delta_g|\eta|^2_g\cdot|\eta|_g^2\cdot \omega^n &= 2m\int_M\langle \sqrt{-1}\partial\bar{\partial}|\eta|_g^2, \beta\rangle_g\cdot\omega^n\\
	&\ge  \int_M |\eta|_g^4S_{2m}(x)\omega^n\ge 0.
\end{align*}
Hence, $|\eta|\equiv C$ for some constant $C$. Since $S_{2m}$ is quasi-positive, then $$
0\ge C^4\int_M S_{2m}(x)\omega^n>0,
$$ 
which is a contradiction if $C\neq 0$. 

We only have to deal with the case: $m=1$, which is a little complicated. Since $\mbox{Ric}_3^\perp$ is quasi-positive, we assume that for any $x\in M$, $\mbox{Ric}_3^\perp(x, \tilde{\Sigma})(v, \bar{v})\ge \varphi(x)|v|^2,$ for any $3$-dimensional subspace $\tilde{\Sigma}\subset T_x'M$ and any $v\in \tilde\Sigma$. Here $\varphi(x)$ is a quasi-positive function on $M$. 

Let $u^1=\frac{e^1+e^2}{\sqrt{2}}, u^2=\frac{e^1-e^{2}}{\sqrt{2}}, u^l=e^l$ for $l\ge 3$. Then $\{u^j\}_{j=1}^n$ is another unitary orthogonal frame of $T_x'M$. Let $\Sigma_l=\{u^1, u^2, u^l\}$ for any $l\ge 3$, we get
\begin{equation*}
	R(u^1,\bar{u}^1,u^2,\bar{u}^2)+R(u^1,\bar{u}^1,u^l,\bar{u}^l)\ge \varphi(x).
\end{equation*}
After direct calculations, we get 
\begin{equation}\label{eq:Ric1}
	R_{1\bar{1}1\bar{1}}+R_{2\bar{2}2\bar{2}}-R_{1\bar{2}1\bar{2}}-R_{2\bar{1}2\bar{1}}+2(R_{1\bar{1}l\bar{l}}+R_{2\bar{2}l\bar{l}}+R_{1\bar{2}l\bar{l}}+R_{2\bar{1}l\bar{l}})\ge 4\varphi(x).
\end{equation}
We also have
$$
R(u^2,\bar{u}^2,u^1,\bar{u}^1)+R(u^2,\bar{u}^2,u^l,\bar{u}^l)\ge \varphi(x),
$$
and 
\begin{equation}\label{eq:Ric2}
	R_{1\bar{1}1\bar{1}}+R_{2\bar{2}2\bar{2}}-R_{1\bar{2}1\bar{2}}-R_{2\bar{1}2\bar{1}}+2(R_{1\bar{1}l\bar{l}}+R_{2\bar{2}l\bar{l}}-R_{1\bar{2}l\bar{l}}-R_{2\bar{1}l\bar{l}})\ge 4\varphi(x).
\end{equation}
By summing (\ref{eq:Ric1}) and (\ref{eq:Ric2}) together, we get 
\begin{equation}\label{eq:Ric3}
	R_{1\bar{1}1\bar{1}}+R_{2\bar{2}2\bar{2}}-R_{1\bar{2}1\bar{2}}-R_{2\bar{1}2\bar{1}}+2(R_{1\bar{1}l\bar{l}}+R_{2\bar{2}l\bar{l}})\ge 4\varphi(x).
\end{equation}
If we consider another unitary frame $u^1=\frac{e^1+i e^2}{\sqrt{2}}, u^2=\frac{e^1-i e^2}{\sqrt{2}}, u^l=e^l, l\ge 3$, similarly, we can get 
\begin{equation}\label{eq:Ric4}
	R_{1\bar{1}1\bar{1}}+R_{2\bar{2}2\bar{2}}+R_{1\bar{2}1\bar{2}}+R_{2\bar{1}2\bar{1}}+2(R_{1\bar{1}l\bar{l}}+R_{2\bar{2}l\bar{l}})\ge 4\varphi(x).
\end{equation}
(\ref{eq:Ric3}) and (\ref{eq:Ric4}) imply that
$$
R_{1\bar{1}1\bar{1}}+R_{2\bar{2}2\bar{2}}+2(R_{1\bar{1}l\bar{l}}+R_{2\bar{2}l\bar{l}})\ge 4\varphi(x).
$$
Together with $R_{1\bar{1}2\bar{2}}+R_{1\bar{1}l\bar{l}}\ge \varphi(x)$ and $R_{2\bar{2}1\bar{1}}+R_{2\bar{2}l\bar{l}}\ge \varphi(x)$,  we have 
$$
\sum_{i,j=1}^2R_{i\bar{i}j\bar{j}}+3(R_{1\bar{1}l\bar{l}}+R_{2\bar{2}l\bar{l}})\ge 6\varphi(x).
$$
Then
\begin{align*}
	&\,\sum_{l=3}^n\limits \left(\sum_{i,j=1}^2R_{i\bar{i}j\bar{j}}+3(R_{1\bar{1}l\bar{l}}+R_{2\bar{2}l\bar{l}})\right)\\
	&=(n-5)\sum_{i,j=1}^2R_{i\bar{i}j\bar{j}}+3(R_{1\bar{1}}+R_{2\bar{2}})\ge 6(n-2)\varphi(x).
\end{align*}
Now we consider 
\begin{align}
	&\,	3\cdot\frac{\Delta_g|\eta|^2_g\cdot|\eta|_g^2}{2}+(n-5)
	<\sqrt{-1}\partial\bar{\partial} |\eta|^2_g, \beta>_g\nonumber\\
	&= \frac{(n-2)|\eta|_g^2}{2}\sum_{i=1}^2|\nabla_i\eta|_g^2+\frac{3|\eta|_g^2}{2}\sum_{i=3}^n|\nabla_i\eta|_g^2\label{eq:Ric5}\\
	&\,+\frac{|\eta|_g^2}{2}\left((n-5)\sum_{i,j=1}^2R_{i\bar{i}j\bar{j}}+3(R_{1\bar{1}}+R_{2\bar{2}})\right)\nonumber\\
	&\ge 3(n-2)|\eta|_g^2\varphi(x).\nonumber
\end{align}
After the same arguments as in the proof of (1), one can also get the contradiction. 
Then $h^{2,0}(M)=0$ and $M$ is projective. 

One can follow the proof for $\mbox{Ric}_3^\perp$ in a similar way when  $\mbox{Ric}_3^+$ is quasi-positive and obtain the desired result.

\textbf{ (3) $\mbox{Ric}_p(x)$ is 2-quasi-positive.}

When $p=2$,  $2$-quasi-positive $\mbox{Ric}_2$ is equivalent to the quasi-positive $S_2$. Then $h^{2,0}=0$ and $M$ is projective proved in \cite{Tang}. 

When $p=n=3$, then $m=1$ and $2$-quasi-positive $\mbox{Ric}_3$ is equivalent to the $2$-quasi-positive Ricci curvature. Let $\lambda_1(x), \lambda_2(x), \lambda _3(x)$ be eigenvalues of Ricci curvature at $x$ and $\lambda_1\le \lambda_2\le\lambda_3$. We assume that $\psi(x)=\lambda_1+\lambda_2$ and  $\psi(x)$ is quasi-positive. 
Now (\ref{eq:laplace}) implies that
$$
\Delta_g|\eta|^2_g=\sum_{i=1}^3|\nabla_i\eta|^2+(R_{1\bar{1}}+R_{2\bar{2}})|\eta|^2.
$$
Here $R_{1\bar{1}}+R_{2\bar{2}}\ge \psi(x)$. With the similar arguments as above, one can also prove the result.

In the following, we assume that $n>3$ and $3\le p\le n$. Since $\mbox{Ric}_p$ is $2$-quasi-positive, then $S_{l}$ is quasi-positive when $l\ge 3$. If $m\ge 2$, one can go through the proof similarly to get the contradiction as in \cite{Tang} (one can also refer the proof in part (2)). We only have to deal with the case $m=1$,  that is, $\eta$ is a nonzero $(2, 0)$-form and $\eta\wedge \eta=0$. Let $\Sigma=\{e^1, e^2, e^{i_1},\cdots, e^{i_{p-2}}\}$ for any $3\le i_1<i_2<\cdots <i_{p-2}\le n$, we have 
$$
\mbox{Ric}_p(x, \Sigma)(e^1, \bar{e}^1)+\mbox{Ric}_p(x, \Sigma)(e^2, \bar{e}^2)
$$
is at least quasi-positive, that is,
$$
\sum_{i,j=1}^2 R_{i\bar{i}j\bar{j}}+\sum_{l=1}^{p-2}(R_{1\bar{1}i_l\bar{i_l}}+R_{2\bar{2}i_l\bar{i_l}})
$$
is at least quasi-positive. 

Through direct calculation, we have
\begin{align}
	&\,	C_{n-3}^{p-3}\frac{\Delta_g|\eta|^2_g\cdot|\eta|_g^2}{2}+C_{n-3}^{p-2}
	<\sqrt{-1}\partial\bar{\partial} |\eta|^2_g, \beta>\nonumber\\
	&= C_{n-3}^{p-3}\frac{|\eta|^2}{2}\sum_{i=1}^n|\nabla_i\eta|^2+C_{n-3}^{p-2}\frac{|\eta|^2}{2}\sum_{i=1}^2|\nabla_i\eta|^2\nonumber\\
	&\,+\frac{|\eta|^4}{2}\left(C_{n-3}^{p-3}(\mbox{Ric}_{1\bar{1}}+\mbox{Ric}_{2\bar{2}})+C_{n-3}^{p-2}\sum_{i,j=1}^2\limits R_{i\bar{i}j\bar{j}}\right)\label{eq:Ricperp}\\
	&= C_{n-3}^{p-3}\frac{|\eta|^2}{2}\sum_{i=1}^n|\nabla_i\eta|^2+C_{n-3}^{p-2}\frac{|\eta|^2}{2}\sum_{i=1}^2|\nabla_i\eta|^2\nonumber\\
	&\,+\frac{|\eta|^4}{2} \sum_{|I|=p-2}\left(\sum_{i,j=1}^2R_{i\bar{i}j\bar{j}}+\sum_{l=1}^{p-2}(R_{1\bar{1}i_l\bar{i_l}}+R_{2\bar{2}i_l\bar{i_l}})\right),\nonumber
\end{align}
where $I=(i_1, i_2,\cdots, i_{p-2}), i_l\in\{3,4,\cdots, n\}$ and $i_1<i_2<\cdots<i_{p-2}$. Here the curvature term in the last equality is at least quasi-positive by the assumption of the curvature condition. Then after integration over $M$, one can argue similarly as above to get the contradiction. Therefore, we have $h^{2, 0}=0$ and $M$ is projective. Theorem \ref{thm1} follows.

\subsection{Proof of Theorem \ref{thm2}} We can take the similar arguments in \cite{Nihol} to prove Theorem \ref{thm2}.

With the same assumptions as in the proof of Theorem \ref{thm1},

\textbf{ $(a)$} if $S_2^\perp\ge 0$, from (\ref{eq:key}) and (\ref{eq:S1}), we can get that
\begin{align*}
	0&\ge-\frac{2m-1}{4m(2m+1)}\int_M |d|\eta|_g^2|_g^2\omega^n\\
	&\ge \frac{(2m-1)}{2m(2m+1)}\int_M  |\eta|_g^2|\nabla \eta|_g^2\omega^n+
	\int_M \frac{|\eta|_g^4}{2m}S_{2m}^\perp(x)\omega^n.
\end{align*}
Then $|\eta|\equiv C$ where $C$ is a constant, and $\nabla\eta\equiv 0$, that is, $\eta$ is parallel. Since $\mbox{Hol}(M, g)=U(n)$, then from the representation theory, we have $\eta\equiv 0$. Therefore, $h^{2, 0}=0$ and $M$ is projective. 

The proof for $S_2^+\ge 0$ is essentially the same as that for $S_2^\perp\ge 0$.

\textbf{$(b)$} if $\mbox{Ric}_3^\perp(x)\ge 0$, then $S_{l}(x)$ is non-negative for $l\ge 3$. When $m=1$, from (\ref{eq:key}) and (\ref{eq:Ric5}), it is easy to get that $|\eta|\equiv C$ for some constant $C$ and $\nabla\eta\equiv 0$ as in the proof of $(a)$. 

When $m\ge 2$, $S_{2m}(x)$ is non-negative. Then from  (\ref{eq:partial}) and (\ref{eq:key}), we can similarly get $|\eta|\equiv C$ for some constant $C$, $\nabla_i\eta=0$ for $1\le i\le 2m$ and $\sum_{i,j=1}^{2m} R_{i\bar{i}j\bar{j}}=0$. In the following, we follow the arguments in \cite{Nihol}. Let $\Sigma=\mbox{span}\{e_i\}_{i=1}^{2m}$. Then $S_{2m}(x,\Sigma)$ attains its minimum among all $2m$-dimensional subspaces in $T_x'M$. Then by the second variation argument (cf. Proposition 4.2 of \cite{Nicrell}), we have that
$$
\sum_{i=1}^{2m}R_{i\bar{i}l\bar{l}}\ge 0\qquad \mbox{for}\quad l\ge 2m+1. 
$$ 
Let $v=e_l$ in (\ref{eq:bochner}), we can  get $\nabla_l \eta=0$ for $l\ge 2m+1$. Hence, we have $\nabla\eta=0$. Then we also have $\eta\equiv 0$ as in $(a)$, which implies that  $h^{2, 0}=0$ and $M$ is projective.

If $\operatorname{Ric}_3^{+}(x) \geq 0$, one can follow the similar arguments presented above to prove the desired result.

\textbf{$(c)$} if $\mbox{Ric}_p$ is 2-non-negative, we will divide into two cases. 

When $m=1$, by (\ref{eq:key}) and (\ref{eq:Ricperp}), it is easy to get that  $|\eta|\equiv C$ for some constant and $\nabla \eta=0$. 

When $m\ge 2$. Since $S_l(x)$ is non-negative when $l\ge 3$, by the proof of (3) in Theorem \ref{thm1} and the argument in the proof of $(b)$ above, we can still get $|\eta|\equiv C$ and $\nabla\eta=0$.  In a word,  $\eta$ is a parallel $(2m, 0)$-form. Then $\eta\equiv 0$ which implies that  $h^{2, 0}=0$ and $M$ is projective.  Theorem \ref{thm2} follows.

\noindent\textbf{ Acknowledgements} The authors would like to thank Professor Lei Ni to put up the questions for us and helpful discussions. The second author is grateful to  Professor Lei Ni and Professor Qingzhong Li for constant encouragement and support. 

\bibliographystyle{amsplain}

\begin{thebibliography}{10}
	\bibitem{BT} K. Broder and K. Tang, \textit{On the weighted orthogonal Ricci curvature}, J. Geom. Phys. \textbf{193}(2023), Paper No. 104783, 13pp.
\bibitem{Cam} F. Campana, \textit{Connexit\'e rationnelle des vari\'et\'es de Fano}, Ann. Sci. \'Ecole Norm. Sup. (4) \textbf{25} (1992), no. 5, 539-545.
\bibitem{CLT} J. Chu, M.-C. Lee and L.-F.Tam, \textit{K\"ahler manifolds and mixed curvatures}, Trans. Amer. Math. Soc. \textbf{375} (2022), no. 11, 7925-7944.
\bibitem{CLZ} J. Chu, M.-C. Lee and J. Zhu, \textit{On K\"ahler manifolds with non-negative mixed curvature}, arXiv:2408.14043.
\bibitem{Demailly} J.-P. Demailly, \textit{Complex analytic and differential geometry}, https://www-fourier.ujf-grenoble.fr/$\sim$demailly/manuscripts/agbook.pdf.
\bibitem{HW} G. Heier and B. Wong, \textit{On projective K\"ahler manifolds of partially positive curvature and rational connectedness,} Doc. Math. \textbf{25} (2020), 219-238.
\bibitem{Kodaira} K. Kodaira, \textit{On K\"ahler varieties of restricted type (an intrinsic characterization of algebraic varieties)}, Ann. of Math. (2) \textbf{60}, (1954),28-48.
\bibitem{KMM} J. Koll\'ar, Y. Miyaoka and S. Mori, \textit{Rational connectedness and boundedness of Fano manifolds}, J. Differential Geom. \textbf{36}(1992), no.3, 765-779.
\bibitem{NiCPAM} L. Ni, \textit{Liouville theorems and a Schwarz lemma for holomorphic mappings between K\"ahler manifolds}, Comm. Pure Appl. Math., \textbf{74} (2021), 1100-1126.
\bibitem{Nicrell} L. Ni, \textit{The fundamental group, rationally connectedness and the positivity of K\"ahler manifolds}, J. Reine Aangew. Math. \textbf{774} (2021), 267-299.
\bibitem{Nihol} L. Ni, \textit{Holonomy and the Ricci curvature of complex Hermitian manifolds}, arXiv:2410.06411s\\v1.
\bibitem{NZ-CV} L. Ni and F. Zheng, \textit{Comparison and vanishing theorems for K\"ahler manifolds}, Calc. Var. Partial Differential Equations \textbf{57} (2018), no. 6, Paper No. 151.
\bibitem{NZ-GT} L. Ni and F. Zheng, \textit{Positivity and the Kodaira embedding theorem}, Geom. Topol. \textbf{26}, No. 6, 2491-2505 (2022).
\bibitem{Tang} K. Tang, \textit{Quasi-positive curvature and vanishing theorems}, arXiv: 2405.03895. 
\bibitem{YangRC} X. Yang, \textit{RC-positivity, rationally connectedness and Yau's conjecture}, Camb. J. Math. \textbf{6}(2018), no. 2, 183-212. 
\bibitem{Yang20} X. Yang, \textit{Compact K\"ahler manifolds with quasi-positive second Chern-Ricci curvature, preprint}, arXiv:2006.13884.
\bibitem{Yau} S. -T. Yau, \textit{ Problem section.} In seminar on Differential Geometry, volume 102 of Ann. of Math Stud. 669-706. 1982.
\bibitem{ZZ} S. Zhang and X. Zhang, \textit{ Compact K\"ahler manifolds with quasi-positive holomorphic sectional curvature}, arXiv: 2311.18779v4.

%

\end{thebibliography}

\end{document}